\documentclass[11pt]{amsart}
\usepackage{latexsym}
\usepackage{amsmath,amssymb,amsfonts}
\usepackage{graphics}
\usepackage[all]{xy}

\def\Bbb{\mathbb}

\def\eea{\end{eqnarray*}}

\newtheorem{defn}{Definition}
\newtheorem{thm}{Theorem}[section]
\newtheorem{prop}[thm]{Proposition}
\newtheorem{cor}[thm]{Corollary}
\newtheorem{lem}[thm]{Lemma}

\newtheorem{remark}[thm]{Remark}
\newtheorem{question}[thm]{Question}

\def\bea{\begin{eqnarray*}}
\def\eea{\end{eqnarray*}}
\def\be{\begin{equation}}
\def\ee{\end{equation}}

\def\k{K\"{a}hler  }

\begin{document}
\renewcommand{\theequation}{\thesection.\arabic{equation}}

\title[Almost-K\"{a}hler Scalar Curvature Functions]{Scalar Curvature Functions of Almost-K\"{a}hler Metrics}

\author{Jongsu Kim and Chanyoung Sung}
\date{\today}

\address{Dept. of Mathematics, Sogang University, Seoul, Korea}
\email{jskim@sogang.ac.kr}

\address{Dept. of mathematics education, Korea National University of Education}
\email{cysung@knue.ac.kr}

\thanks{The first author was supported by the National Research Foundation of Korea(NRF) grant funded by the Korea government(MOE) (No.NRF-2010-0011704).
}

\keywords{almost-K\"ahler, scalar curvature, symplectic manifold}

\subjclass[2010]{53D05,53D35,53C15,53C21}

\begin{abstract}
For a closed smooth manifold $M$ admitting a symplectic structure, we define a smooth topological invariant $Z(M)$ using {\it almost-K\"{a}hler} metrics, i.e. Riemannian metrics compatible with symplectic structures. We also introduce $Z(M, [[\omega]])$ depending  on symplectic deformation equivalence class $[[\omega]]$.
 We first prove that there exists a 6-dimensional smooth manifold $M$ with more than one deformation equivalence classes with different signs of $Z(M, [[\omega]] )$. Using $Z$ invariants, we set up a Kazdan-Warner type problem of classifying symplectic manifolds into three categories.

 We finally prove that on every closed symplectic manifold $(M, \omega)$ of dimension $\geq 4$, any smooth function which is somewhere negative and somewhere zero can be the scalar curvature of an almost-K\"ahler metric compatible with a symplectic form which is deformation equivalent to $\omega$.
\end{abstract}

\maketitle

\setcounter{section}{0}
\setcounter{equation}{0}

\section{Introduction}
In Riemannian geometry, scalar curvature encodes certain information of the differential topology of a smooth closed manifold. There has been much progress on topological conditions for the existence of a metric of positive scalar curvature, and more generally  which functions on a given manifold can be the scalar curvature of a Riemannian metric.

The latter problem is known as the Kazdan-Warner problem. They proved \cite{kw3} that
the necessary and sufficient condition for a smooth function $f$ on  a closed manifold $M$ of dimension $\geq 3$ to be the scalar curvature of some metric is
\begin{itemize}
     \item $f$ is arbitrary, in case $Y(M)>0$,
     \item $f$ is identically zero or somewhere negative, in case $Y(M)=0$ and $M$ admits a scalar-flat metric,
     \item $f$ is negative somewhere, in the remaining case,
\end{itemize}
where $Y(M)$ denotes the Yamabe invariant of $M$. For the Yamabe invariant, the readers are referred to \cite{koba}.

\smallskip
As an extension of the Kazdan-Warner problem, it is natural to pursue a similar classification in some restricted class of Riemannian metrics.

Let $M$ be a smooth manifold with a symplectic form $\omega$.
An almost-complex structure $J$ is called $\omega$-compatible, if
$\omega(J\cdot,J\cdot)=\omega(\cdot,\cdot)$, and $\omega(\cdot,J\cdot)$ is positive-definite. Thus a
smooth $\omega$-compatible $J$ defines a smooth $J$-invariant Riemannian metric   $g(\cdot,\cdot):=\omega(\cdot,J\cdot)$, which is called an $\omega$-almost-K\"ahler metric. It is K\"ahler iff $J$ is integrable.

Due to the lack of a Yamabe-type theorem (\cite{LP}) which would produce metrics of constant scalar curvature, a Kazdan-Warner type result for symplectic manifolds is left open even without any conjectures. A general existence result so far is in \cite{Ki2} stating that every symplectic manifold of dimension $\geq 4$ admits a complete compatible almost-K\"ahler metric of negative scalar curvature.

In \cite{KS, kim}, we studied a symplectic version of the Kazdan-Warner problem and in particular we characterized the scalar curvature functions of some symplectic tori and nil-manifolds.
Here we shall refine this version of Kazdan-Warner problem.

\smallskip

Let us recall that
two symplectic forms $\omega_0$ and $\omega_1$ on $M$ are called {\it deformation equivalent},
 if there exists a diffeomorphism $\psi$ of $M$  such that  $\psi^* \omega_1$ and $\omega_0$ can be joined by a smooth homotopy of sympelctic forms, \cite{MS}.
 There are a number of smooth manifolds which admit more than one deformation equivalence classes: see \cite{S} or references therein.
For a symplectic form $\omega$, its  deformation equivalence class shall be denoted by $[[\omega]]$.
By abuse of notation, we say that a metric $g$ is in $[[\omega]]$ when $g$ is compatible with a symplectic form $\omega$ in $[[\omega]]$.

\begin{defn}
Let  $M$ be a smooth closed manifold of dimension $2n\geq 4$ which admits a symplectic structure.
Define $$Z(M, [[\omega]] ) = \sup_{g \in [[\omega]]}  \frac{\int_M s_g  d\textrm{vol}_g}{ (\textrm{Vol}_g)^{\frac{n-1}{n}}},$$  where $s_g$ is the scalar curvature of $g$, and  define $$Z(M) = \sup_{[[\omega]]}  Z(M, [[\omega]]).$$
\end{defn}
The denominator in $\frac{\int_M s_g  d\textrm{vol}_g}{ (\textrm{Vol}_g)^{\frac{n-1}{n}}}$ was put for the invariance under a scale change $\omega\rightarrow c\cdot\omega$ with $c>0$, and one can get the following inequality from the formulas (\ref{sminus}) and (\ref{Blair}) below;
\begin{equation} \label{zmo}
Z(M,[[\omega]])\leq\sup_{\omega\in [[\omega]]}\frac{ 4\pi c_1(\omega)
\cdot \frac{[\omega]^{n-1}}{(n-1)!}}{(\frac{[\omega]^n}{n!})^{\frac{n-1}{n}}},
\end{equation}
\noindent where $c_1 (\omega)$ is the first Chern class of $\omega$.

\smallskip
These $Z$ numbers may take the value of $\infty$ and are different in nature from the Yamabe invariant which is bounded above in each dimension. Obviously $Z(M)$ is a smooth topological invariant of $M$.
Although the quantity in the right hand side of (\ref{zmo}) may serve usefully for many purposes, the $Z$ value reflects almost-\k geometry better, so seems more relevant to our purpose.
Of course, it would be very interesting to know if the equality in (\ref{zmo}) always holds or not.

To explain why possible scalar curvature functions may depend on $[[\omega]]$, we shall demonstrate a smooth manifold which admits two symplectic deformation equivalence classes with distinct signs of $Z(M, [[\omega]])$.
\begin{thm} \label{main2}
 There exists a smooth closed 6-dimensional manifold with distinct symplectic deformation equivalence classes $[[\omega_i]]$, $i=1,2$ such that $Z(M, [[\omega_1]] ) =\infty $ and $Z(M, [[\omega_2]] ) <0$ .
\end{thm}


Next, we use the method of \cite{Ki2, kim} to prove our main theorem;
\begin{thm} \label{maint}
Let $(M, [[\omega]])$ be a smooth closed manifold of dimension $2n\geq 4$ with a  deformation equivalence class of symplectic forms. Then any smooth function on $M$  which is somewhere negative and somewhere zero is the scalar curvature of some smooth almost-K\"ahler metric in $[[\omega]]$.
\end{thm}

We speculate that any smooth somewhere-negative function on a closed manifold $M$ with $[[\omega]]$ might be the scalar curvature of some smooth almost-K\"ahler metric in $[[\omega]]$.
 A key step to prove it should be to show that every $(M, [[\omega]])$ has an almost-K\"ahler metric of negative constant scalar curvature in $[[\omega]]$.




With $Z$ ready, one can see that Theorem \ref{maint} contributes to answering the following question;
\begin{question} \label{cj1}
Let  $M$ be a smooth closed manifold of dimension $2n\geq 4$ admitting a symplectic structure.

 Is the (necessary and sufficient) condition for a smooth function $f$ on $M$ to be the scalar curvature of some smooth almost-K\"ahler metric as follows?

  \ \     {\rm (a)} $f$ is arbitrary, if $0<Z(M) \leq \infty$,

  \ \     {\rm (b)} $f$ is identically zero or somewhere negative, if $Z(M)=0$ and $M$ admits a scalar-flat almost-K\"ahler metric,

  \ \      {\rm (c)} $f$ is negative somewhere, if otherwise.

\smallskip
Also, is the condition for a smooth function $f$ on $M$ to be the scalar curvature of some smooth almost-K\"ahler metric in $[[\omega]]$  as follows?

      \ \      {\rm ($a^{'}$)}  $f$ is arbitrary,  if $0<Z(M, [[\omega]]) \leq \infty$,

      \ \      {\rm ($b^{'}$)}  $f$ is identically zero or somewhere negative, if $Z(M, [[\omega]])=0$ and $M$ admits a scalar-flat almost-K\"ahler metric in $[[\omega]]$,

      \ \      {\rm ($c^{'}$)}  $f$ is negative somewhere, if otherwise.

\end{question}

This paper is organized as follows. In section 2, some computations of $Z(M)$ are explained. Theorem 1.1 is proved in section 3.
In section 4, Kazdan-Warner type argument  in almost \k setting is explained.
Theorem 1.2 is proved in section 5 and 6.


\section{Some computations of $Z(M)$}

In this section we explain some basic properties and computations of $Z$ invariant in relatively simple cases.

\begin{lem}\label{sung-prop}
Let $M$ be a smooth closed  manifold of dimension $\geq 4$ admitting a symplectic structure. 
If $Y(M)\leq 0$, then $Z(M)\leq 0$.
\end{lem}
\begin{proof}
Suppose not. Then there exists an almost K\"ahler metric $g$ on $M$ such that $\int_M s_gdvol_g>0$. Then by the Yamabe problem \cite{LP}, a conformal change of $g$ gives a metric of positive constant scalar curvature. This implies that $Y(M)>0$, thereby yielding a contradiction.
\end{proof}
\begin{lem}
Any compact minimal K\"ahler surface $M$ of Kodaira dimension 0 has $Z(M)=0$, and it is attained by a Ricci-flat \k metric.
\end{lem}
\begin{proof}
By the Kodaira-Enriques classification \cite{BPV}, such $M$ is K3 or $T^4 $ or their finite quotients, and hence it admits a Ricci-flat K\"ahler metric. Thus $Z(M)\geq 0$. Since $M$ cannot admit a metric of positive scalar curvature, (in fact, $Y(M)=0$), the above lemma forces $Z(M)=0$.
\end{proof}

By using Lemma \ref{sung-prop}, one can show that if $Y(M)=0$, and $M$ admits a ``collapsing" sequence of almost-K\"aher metrics with bounded scalar curvature, then $Z(M)=0$. For example, we consider
the Kodaira-Thurston manifold  $M_{KT}$; see Section 4.
\begin{lem}
For the Kodaira-Thurston manifold  $M_{KT}$, $Z(M_{KT})=0$, and it is obtained as the limit by a collpasing sequence of almost-\k metrics of negative constant scalar curvature.
\end{lem}
\begin{proof}
$M_{KT}$ never admits a metric of positive scalar curvature;  one can use Seiberg-Witten theory, see \cite{OO}. So, $Y(M_{KT})\leq 0$. Thus $Z(M_{KT})\leq 0$ by Lemma \ref{sung-prop}.

We let $\Bbb R^4=\{(x,y,z,t)\}$ endowed with the metric $dx^2 + dy^2 + (dz - x dy)^2 + dt^2$. Then $M_{KT}$ with an almost-K\"ahler metric is obtained as the quotient of $\Bbb R^4$ by the group generated by isometric actions,
$$\gamma_1(x,y,z,t)=(x+1,y,y+z,t),\ \ \ \ \gamma_2(x,y,z,t)=(x,y+1,z,t),$$
$$\gamma_3(x,y,z,t)=(x,y,z+1,t),\ \ \ \ \gamma_4(x,y,z,t)=(x,y,z,t+d),$$ where $d$ is any positive constant.

For any $d>0$, the scalar curvature is $-\frac{1}{2}$. (See the curvature computations in Lemma \ref{Rhu-win}.) But by taking $d>0$ sufficiently small, we can get an almost-K\"ahler metric $g$ on $M_{KT}$ such that
$\int_{M_{KT}} s_g  d\textrm{vol}_g/ (\textrm{Vol}_g)^{\frac{1}{2}}$ is arbitrarily close to 0. Therefore  $Z(M_{KT})=0$.
\end{proof}

Now $M_{KT}$, with  $Z(M_{KT})=0$, never admits a scalar-flat almost-K\"ahler metric, because such a metric has to be a K\"ahler metric, which is not allowed on $M_{KT}$ with $b_1(M_{KT})=3$. Therefore one can expect that  $M_{KT}$ belongs to the category (c) in the classification of Question \ref{cj1}. Indeed this was already proved in \cite{kim}. Note that $M_{KT}$ is a nilmanifold. One may expect to find more examples of symplectic solvmanifolds with vanishing $Z$ value and
prove them to be in the category (c).



 Now let us give an example of $Z(M)>0$.
\begin{lem}
For the complex projective plane $\Bbb CP^2$, $Z(\Bbb CP^2)=12\sqrt{2}\pi$, and it is attained by a \k Einstein metric.
\end{lem}
\begin{proof}
First we claim that $\Bbb CP^2$ with the reversed orientation does not support any almost complex structure. Suppose it does. Recall that any closed almost complex 4-manifold satisfy
\begin{eqnarray*}
c_1^2=2\chi+3\tau,
\end{eqnarray*}
where $\chi$ and $\tau$ respectively denote Euler characteristic and signature. Then this formula
gives $c_1^2=3$, which is impossible because of the fact that $c_1\in H^2(\Bbb CP^2;\Bbb Z)$. Thus  $c_1^2$ must be 9 so that $c_1$ is  $3[H]$  or $-3[H]$, where $[H]$ is the hyperplane class.

For any symplectic form $\omega$ on $\Bbb CP^2$, $[\omega]$ must be a nonzero multiple of $[H]$. We apply the Blair formula (\ref{Blair}) to get
\begin{eqnarray*}
  \frac{\int_{\Bbb CP^2} s_gdvol_g}{(\textrm{Vol}_g)^{\frac{1}{2}}} &\leq&\frac{\int_{\Bbb CP^2}  \frac{1}{2} (s_g + s^{*}_g)\ dvol_g}{(\textrm{Vol}_g)^{\frac{1}{2}}}\\ &=&  \frac{4\pi c_1(\omega)\cdot [\omega]}{(\frac{[\omega]\cdot[\omega]}{2})^{\frac{1}{2}}}\\
&=& \pm 12\sqrt{2}\pi,
\end{eqnarray*}
for any $\omega$-almost-K\"ahler metric $g$ on $\Bbb CP^2$.
In the above inequality, we used a relation between $s$ and the star-scalar curvature $s^*$  \cite{AD};
\begin{equation} \label{sminus}
  s^* -s =\frac{1}{2} |\nabla J|^2 \geq 0.
\end{equation}

Therefore $Z(\Bbb CP^2,[[\omega]])\leq 12\sqrt{2}\pi$. In fact the Fubini-Study metric with $\omega_{FS}$ saturates this  upper bound, so we finally get
$$Z(\Bbb CP^2)=Z(\Bbb CP^2,[[\omega_{FS}]])=12\sqrt{2}\pi.$$
\end{proof}

\begin{remark}
{\rm  In the above we only treated a few simple examples.
However, we expect that $Z$ invariant is fairly computable under some symplectic surgeries. For instance, one can see that
$Z( \hat{T}^4 )=0$, where $\hat{T}^4 $ is the blow-up at one point of the \k $4$-torus. In fact,
one only needs to check that LeBrun's argument in the proof of theorem 3 of \cite{Le3}
still works in almost \k context. This gives $Z( \hat{T}^4 ) \geq 0$.
Together with Seiberg-Witten theory one gets $Z( \hat{T}^4 ) = 0$. A similar argument, albeit in \k case, may be found in \cite{SS}.}
\end{remark}

\section{Symplectic deformation classes on a manifold with distinct signs of $Z(\cdot,  [[\omega]])$}

\noindent In this section we shall prove Theorem \ref{main2}.

We use one of the examples in \cite{Ru}. Let $W$  be a complex Barlow surface, which is a minimal complex surface of general type homeomorphic, but not diffeomorphic, to $R_8$, the blown-up complex surface at 8 points in general position in the complex projective plane. By a small deformation of complex structure we may assume that $W$ has ample canonical line bundle \cite{CL}. Then by Yau's solution of Calabi conjecture,  $W$ and $R_8$ admit a \k Einstein metric of negative (and positive, respectively) scalar curvature. Ruan showed that for a compact Riemann surface $\Sigma$,
  $R_8 \times \Sigma$ and $W \times \Sigma$ are diffeomorphic but their natural symplectic structures are not deformation equivalent.

  We prove;

\begin{prop} \label{riem} Let $W$ be a Barlow surface with ample canonical line bundle  and $\Sigma$ be a Riemann surface of genus 2. Consider a \k Einstein  metric of negative scalar curvature on $W$ with \k form $\omega_W$ on $W$ and a \k form $\omega_{\Sigma}$ on $\Sigma$ with constant negative scalar curvature.

Then $Z(W \times \Sigma, [[\omega_W + \omega_{\Sigma}]])= -12\pi $, and it is attained by a \k Einstein metric.
\end{prop}

\begin{proof}  We recall a few results about $W$ from \cite[Section 4]{Ru};
there is a homeomorphism of $W$ onto $R_8$ which preserves the Chern class $c_1$
and there is a diffeomorphism of $W \times \Sigma$ onto $R_8 \times \Sigma$ which preserves $c_1$.

Then, the first Chern class of $W$ can be written as
 $c_1(W) = 3E_0 - \sum_{i=1}^{8} E_i \in  H^2(W, \mathbb{R}) \cong \mathbb{R}^9$, where  $E_i$, $i=0, \cdots 8$, is the Poincare dual of a homology class $\tilde{E}_i $, $i=0, \cdots 8$  so that
  $\tilde{E}_i$, $i=0, \cdots 8$, form a basis of $H_2(W, \mathbb{Z}) \cong \mathbb{Z}^9$ and their intersections satisfy $\tilde{E}_i \cdot \tilde{E}_j = \epsilon_i \delta_{ij}$, where $\epsilon_0 =1$
  and $\epsilon_i =-1$ for $i \geq 1$.
 So, in this basis the intersection form
 becomes
$$I=
\left[
  \begin{array}{cccc}
    1 & 0 &  \cdot  \ \ \  \cdot & 0 \\
    0 & -1 & \cdot \  \ \ \cdot & 0 \\
    . & . & \cdot \ \ \ \cdot & 0 \\
     . & . & \cdot \ \ \ \cdot & 0 \\
    0 & 0 & 0 & -1 \\
  \end{array}
\right].
$$


  We have the orientation of $W$ induced by the complex structure and the fundamental class $[W] \in H_4(W, \mathbb{Z})\cong \mathbb{Z}$. As $\omega_W$ is \k Einstein of negative scalar curvature, we may get $[\omega_W] = -3E_0 +  \sum_{i=1}^{8}  E_i$ by scaling if necessary.


\medskip
Now a compact Riemann surface $\Sigma$ of genus $2$ has its fundamental class $[\Sigma] \in H_2(\Sigma, \mathbb{Z}) \cong \mathbb{Z}$. Let $c$  be the generator of
$H^2(\Sigma, \mathbb{Z}) \cong \mathbb{Z}$ such that the pairing $\langle c, [\Sigma] \rangle=1$. Then $c_1(\Sigma) = -2c$.
We consider a \k form $\omega_h$ with constant negative scalar curvature such that $[\omega_h]= c \in H^2(\Sigma, \mathbb{R}) \cong \mathbb{R}$.

 \bigskip
Set $M = W \times \Sigma$.
By K\"{u}nneth theorem, $H^2(M, \mathbb{R})   \cong \pi_1^* H^2(W) \oplus \pi_2^* H^2(\Sigma) \cong \mathbb{R}^9 \oplus  \mathbb{R}$, where $\pi_i$ are the projection of $M$ onto the i-th factor.
 Then,
 $$c_1(M) =  \pi_1^* c_1(W) +  \pi_2^* c_1(\Sigma) =  \pi_1^*( 3E_0  - \sum_{i=1}^{8} E_i  )   -2  \pi_2^*c  \in H^2(M, \mathbb{R}).$$
   Consider any smooth path of symplectic forms $\omega_t$, $0 \leq t \leq \delta$,  on $M$ such that $\omega_0 = \omega_W + \omega_h$.
   We may write
   $$[\omega_t] = n_0 (t) \pi_1^* E_0+ \sum_{i=1}^{8} n_i (t) \pi_1^* E_i  + l(t)  \pi_2^*c  \in H^2(M, \mathbb{R})$$
   for some smooth functions $n_i (t), l(t)$, $i=0, \cdots, 8$.  As they are connected, their first Chern class $c_1(\omega_t)= c_1(M)$. We have;
    \begin{eqnarray} \label{omet3}
    \ \ \ \ \ \ \ \  [\omega_t]^3 ([W \times \Sigma]) & = [  n_0 (t) \pi_1^* E_0+ \sum_{i=1}^{8} n_i (t) \pi_1^* E_i  + l(t)  \pi_2^*c]^3([W \times \Sigma])  \\
  &= 3 \{ n_0^2(t)  -   \sum_{i=1}^{8} n_i^2(t) \} l(t) >0. \ \ \ \   \ \ \ \ \ \ \   \ \ \ \ \ \ \  \ \ \ \   \ \ \ \nonumber
    \end{eqnarray}
    As $l(0) =1>0$, $l(t) >0$. So, $ n_0^2(t)  >   \sum_{i=1}^{8} n_i^2(t)$.  From above, we know that $n_0(0)=-3<0$. As $n_0^2(t)>  0$, we get $n_0(t)<0$.
   \begin{equation} c_1 \cdot [\omega_t]^2 ([W \times \Sigma]) = -2 \{ n_0^2(t) - \sum_{i=1}^{8} n_i^2(t) \}  + 2 l(t) \{ 3 n_0(t) + \sum_{i=1}^{8} n_i(t)   \}.
   \end{equation}

  \noindent Since $n_0^2(t)>  \sum_{i=1}^{8}  n_i^2(t)$ and  $| \sum_{i=1}^{8} n_i (t) | \leq \sqrt{8} \sqrt{ \sum_{i=1}^{8} n_i^2}$, we get
    \begin{eqnarray} \label{nit}
     \ \ \ \ 3 n_0(t) + \sum_{i=1}^{8} n_i(t) & \leq 3 n_0(t) +   2 \sqrt{2} \sqrt{ \sum_{i=1}^{8} n_i^2(t)}  \ \ \ \  \ \ \ \ \ \ \ \ \ \ \ \ \ \  \ \ \  \\
    & < 3 n_0(t) +2 \sqrt{2} \sqrt{n_0^2(t)}
     = (3-2\sqrt{2}) n_0 (t) <0. \nonumber
    \end{eqnarray}

So, $c_1 \cdot [\omega_t]^2 ([W \times \Sigma]) <0$.
Putting $A=  n_0^2(t) - \sum_{i=1}^{8} n_i^2(t)$ and $B=3 n_0(t) + \sum_{i=1}^{8} n_i(t)$,  we have
$$\frac{c_1 [\omega_t]^2}{ [\omega_t^3]^{2/3}} =\frac{2}{3^{2/3}}  \{ \frac{  -A  +  l(t) B}{A^{2/3} l(t)^{2/3} } \}= \frac{2}{3^{2/3}} \{ \frac{  -A^{1/3} }{ l(t)^{2/3} }      +   \frac{   l(t)^{1/3} B}{A^{2/3} }  \}. $$

 For $a, b<0$ and $ x>0$,  set $h(x):=\frac{2}{3^{2/3}} (\frac{a}{x^2} + \frac{x}{b})$.
 Since $h^{'} (x) = \frac{2}{3^{2/3}} (\frac{x^3 - 2ab}{x^3 b})$,
 $h(x)$ has maximum when $x= (2ab)^{1/3}$.
So we get
 $h(x) \leq   (\frac{6a}{b^2})^{\frac{1}{3}}$.
With $a= -A^{1/3} $,  $b= \frac{A^{2/3}}{B} $ and  $x= l(t)^{1/3} $,   this gives
 $$ \frac{c_1 [\omega_t]^2}{ [\omega_t^3]^{2/3}}  \leq  - 6^{\frac{1}{3}}(\frac{B^2}{A})^{\frac{1}{3}} ,$$ and
 from (\ref{nit}) $$\frac{B^2}{A} \geq  \frac{\{3 n_0(t) +2 \sqrt{2} \sqrt{ \sum_{i=1}^{8} n_i^2(t)}\}^2}{n_0^2(t)  -   \sum_{i=1}^{8} n_i^2(t)} = \frac{(3- 2\sqrt{2} \sqrt{y})^2}{1-y}$$ where
 $y=   \sum_{i=1}^{8} \frac{n_i^2(t)}{n_0^2(t)}$.
 By calculus, $\frac{(3- 2\sqrt{2} \sqrt{y})^2}{1-y} \geq 1$ for $y \in [0,1)$ with equality at $y= \frac{8}{9}$.

We have
 $ \frac{c_1 [\omega_t]^2}{ [\omega_t^3]^{2/3}}  \leq  - 6^{\frac{1}{3}} $; the equality is achieved exactly when $n_0(t) =-3$, $n_i(t) =1$, $i=1, \cdots, 8$ and $l(t)=2$ modulo scaling, i.e. when $[\omega_t]$ is a positive multiple of $-c_1(M)$. The \k form of a product \k Einstein metric
 on $M= W \times \Sigma$ belongs to this class.

As the expression $ \frac{ 4\pi c_1(\omega)
\cdot \frac{[\omega]^{n-1}}{(n-1)!}}{(\frac{[\omega]^n}{n!})^{\frac{n-1}{n}}}$ is invariant under a change $\omega \mapsto \phi^*(\omega)$ by a diffeomorphism $\phi$,  so from (\ref{zmo}),
$$Z(M, [[\omega_0]] )  \leq  \sup_{\omega \in [[\omega_0]]} 2\pi \cdot 6^{2/3} \frac{c_1 [\omega]^2}{[\omega^3]^{2/3}} \leq -12\pi.  $$  As the equality is attained by a \k Einstein metric, $Z(M, [[\omega_0]] ) = -12\pi  $.
\end{proof}

The next corollary yields Theorem \ref{main2}.
\begin{cor} There exists a smooth closed 6-d manifold with distinct symplectic deformation equivalence classes $[[\omega_1]]$  and $[[\omega_2]]$ such that $Z(M, [[\omega_1]] ) =\infty $ which is obtained by a sequence of \k metrics of positive constant scalar curvature, and $Z(M, [[\omega_2]] ) <0$ .
\end{cor}

\begin{proof}
Consider $V \times \Sigma$ and $W \times \Sigma$ where $V=R_{8}$,
 $ \  W$ and  $\Sigma$ are as in Proposition \ref{riem}.
 They are diffeomorphic but their natural symplectic structures are not deformation equivalent.
 Let $\omega_1$ be the \k form of a product \k Einstein metric on  $V \times \Sigma$.  One can easily get $Z(M, [[\omega_1]] ) =\infty $ by scaling on one factor of the product.
Let  $\omega_2$ be the $\omega_W + \omega_{\Sigma}$ in Proposition \ref{riem}.
\end{proof}

\begin{remark}
{\rm  Theorem \ref{main2} and its proof hint that much more examples may be obtained in a similar way.
Just for another instance, considering the example $M= R_8 \times R_8$ of Catanese and LeBrun in \cite{CL}, we could check that the smooth 8-dimensional manifold $M$ admits distinct symplectic deformation equivalence classes $[[\omega_i]]$, $i=1,2$ such that $Z(M, [[\omega_1]] ) =\infty $ and $Z(M, [[\omega_2]] ) <0$. }
\end{remark}

\section{Kazdan-Warner method adapted to almost \k metrics}
Our method to prove Theorem \ref{maint} is an adaptation of ordinary scalar curvature theory to an almost-K\"ahler setting, and recall and modify the material explained in \cite{KS, kim}.

Let $\frak{M}$ denote the space of Riemannian metrics on a given smooth manifold $M$ of real dimension $2n$, and we regard $\frak{M}$ as the completion of smooth metrics with respect to $L_2^p$-norm for $p> 2n$. Given a symplectic form $\omega$ on $M$, let $\Omega_\omega$ be the subspace of $\omega$-almost-K\"ahler metrics on $M$. The space $\Omega_{\omega}$ is a smooth Banach manifold with the above norm, and its tangent space  $T_g\Omega_\omega$ at $g:=\omega(\cdot,J\cdot)$ consists of symmetric $(0,2)$ tensors $h$ which are $J$-anti-invariant, i.e.  $$h^+(X,Y):=\frac{1}{2}(h(X,Y)+h(JX,JY))=0$$ for all $X,Y\in TM$.

The space $\Omega_{\omega}$ admits  a natural
parametrization by the exponential map; for $g \in \Omega_{\omega}$, define ${\mathcal
E}_g : T_{g} \Omega_{\omega} \rightarrow \Omega_{\omega}$
by
 ${\mathcal E}_g(h) = g\cdot e^h$ with $$g \cdot e^h(X,Y) = g (X, e^{\hat{h}}(Y))=g (X, Y + \sum_{k=1}^{\infty} \frac{1}{k!} \hat{h}^kY),$$ where $X,Y\in TM$, and $\hat{h}$ is the
$(1,1)$-tensor field lifted from $h$ with respect to $g$. Clearly we have $$\frac{d\{ g \cdot e^{th} \}}{dt}|_{t=0} =h.$$
Given $g \in \Omega_{\omega}$ with
corresponding $J$, any other metric $\tilde{g}$ in $
\Omega_{\omega}$ can be expressed as
\begin{equation}
\tilde{g}= g\cdot e^h, \label{geh}
\end{equation}
where $h$ is a $J$-anti-invariant symmetric $(0,2)$-tensor field
uniquely determined.

\bigskip
For the scalar curvature functional $S_\omega:\Omega_\omega\rightarrow L^p(M)$,  the derivative at $g$ is given by
\begin{equation} D_gS_\omega(h)=\delta_g\delta_g(h)-g(r_g,h) \nonumber \end{equation} for $h\in T_g\Omega_\omega$, where $r_g$ is the Ricci curvature of $g$, and its formal adjoint is given by
$$(D_gS_\omega)^*(\psi)=(\nabla d\psi)^--r_g^-\psi,$$ where $A^-$ for a symmetric $(0,2)$ tensor $A$ denotes the $J$-anti-invariant part
$\frac{1}{2}(A(\cdot,\cdot)-A(J\cdot,J\cdot))$.

The followings are key facts for the Kazdan-Warner type problem in the almost-K\"ahler setting.
\begin{lem} \cite{KS}, \cite[Lemma 1]{kim}
If $D_gS_\omega$ is surjective  for a smooth $g \in \Omega_{\omega}$, then $S_\omega$ is locally surjective at $g$, i.e. there exists an $\epsilon > 0$ such that for any $f\in L^p(M)$ with $||f-S_\omega(g)||< \epsilon$ there is an $L_2^p$ almost-K\"ahler metric $\tilde{g} \in  \Omega_{\omega}$ satisfying $f=S_\omega(\tilde{g})$. Furthermore if $f$ is $C^\infty$, so is $\tilde{g}$.
\end{lem}

Recall that a diffeomorphism  $\phi$ is said to be isotopic to the identity map
if there is a homotopy of diffeomorphisms $\phi_t$, $0 \leq t \leq 1$, such that
$\phi_0 = id$ and $\phi_1= \phi$.
The following lemma was proved in \cite[Theorem 2.1]{KW1} without the isotopy clause. Here we add a few arguments in their argument to verify the isotopy part.

\begin{lem}
Suppose  $\dim M\geq 2$ and $f\in C^0(M)$. Then an $L^p$ function $f_1$ on $M$ belongs to the $L^p$ closure of

$\{f\circ \phi \ | \ \phi \textrm{ is a diffeomorphism of }M,  \textrm{ isotopic to the identity map}\}$

\noindent if and only if $ \ \inf f\leq f_1\leq \sup f$ almost everywhere.
\end{lem}

\begin{proof}
Let $\varepsilon >0$ be given. As $C^0(M)$ is dense in $L^p(M)$, we may assume $f_1 \in C^0(M)$.
We triangulate $M$ into  n-simplexes $\Delta_i$ so that $M= \cup \Delta_i$ and that
$\max_{x,y \in \Delta_i} |f_1(x) - f_1(y)| < \delta$
with $2 \delta = \frac{\varepsilon}{(2 {\rm vol}M)^{ 1 \over p}}$. Choose $b_i$ in the interior of $\Delta_i$.

There exist disjoint open balls $V_i \subset M$ such that $| f(y) - f_1(b_i)    | < \delta$ for $y \in V_i$ and for each $i$.
One chooses a neighborhood $\Omega$ of the $(n - 1)$-
skeleton $M^{(n-1)}$ of $M$, disjoint from the $b_i$, so small that
$$ (\max_M |f| + \max_M |f_1|)^p {\rm Vol(\Omega)} < \frac{\varepsilon^p}{2} .$$

For each $i$, let $U_i$ be a small ball neighborhood of $b_i$, disjoint from $\Omega$, and choose open
sets $O_1$ and $O_2$, so that
$$M - \Omega \subset O_1 \subset \bar{O_1} \subset O_2 \subset \bar{O_2} \subset M- M^{(n-1)}.$$

There is a homotopy of diffeomorphisms $\phi_t$, $0 \leq t \leq 1$, such that
$\phi_0 = id$ and $\phi_1 (U_i) \subset  V_i$, for each $i$.
And there is a homotopy of  diffeomorphisms $\psi_t$, $0 \leq t \leq 1$, with
$\psi_0 = id$ such that
$\psi_1$  satisfies $ \psi_1 | M - O_2  = id$  and that $\psi_1 ( O_1 \cap \Delta_i ) \subset U_i$, for each
$i$. Let $\Phi_t = \phi_t \circ \psi_t$. Then, we get
\begin{align*}
  \parallel f \circ \Phi_1 - f_1 \parallel_p^p  &=
(\int_{\Omega}   +  \int_{M- \Omega}) (|f \circ  \Phi_1 - f_1|^p dvol) \\
& < \frac{\varepsilon^p}{2}  + \sum_i \int_{O_1 \cap \Delta_i}   | f \circ  \Phi_1(y) -f_1(b_i) + f_1(b_i)  - f_1(y)  |^p dvol \\
&<\frac{\varepsilon^p}{2} +  \sum_i  2^{p} \delta^p Vol(\Delta_i) = \varepsilon^p.
\end{align*}
This proves if part, and only if part should be clear.
\end{proof}

\begin{prop}\label{mainprop}
If $D_gS_\omega$ is surjective for a smooth $g \in \Omega_{\omega}$, then any smooth function $f$ with $\inf f\leq S_\omega(g)\leq \sup f$ is the scalar curvature of a smooth almost-K\"ahler metric compatible with ${\phi}^*\omega$ for a diffeomorphism ${\phi}$ of $M$ isotopic to the identity.
\end{prop}
\begin{proof}
By the above two lemmas, there exists a diffeomorphism $\tilde{\phi}$ isotopic to the identity such that $f\circ\tilde{\phi}=S_\omega(\tilde{g})$ for a smooth $\omega$-almost-K\"ahler metric $\tilde{g}$. Thus  $f=S_{(\tilde{\phi}^{-1})^*\omega}((\tilde{\phi}^{-1})^*\tilde{g})$.
\end{proof}

\begin{lem}\label{kim-idea}
The principal part of the fourth-order linear partial differential operator $(D_gS_\omega)\circ(D_gS_\omega)^*: C^\infty(M)\rightarrow C^\infty(M)$ is equal to that of $\frac{1}{2}\Delta_g\circ \Delta_g$, where $\Delta_g$ is the $g$-Laplacian.
\end{lem}
\begin{proof}
Let $e_1,Je_1,\cdots,e_n,Je_n$ be a local orthonormal frame satisfying
$Je_{2i-1}=e_{2i}$ for $i=1,\cdots,n$.
The fourth order differentiation only occurs at $\delta_g\delta_g(\nabla d\psi)^-$. A direct computation shows that
\begin{eqnarray*}
\delta_g\delta_g(2\nabla d\psi)^-(\psi)&=&\sum_{i=1}^{2n}\sum_{j=1}^{2n}e_i(e_j(\nabla d\psi(e_i,e_j)-\nabla d\psi(J(e_i),J(e_j))))+\textrm{l.o.t.}\\ &=& \sum_{i=1}^{2n}\sum_{j=1}^{2n}\{e_i(e_j(e_i(e_j\psi)))-e_i(e_j(Je_i(Je_j\psi)))\}+\textrm{l.o.t.}\\ &=& \sum_{i=1}^{2n}\sum_{j=1}^{2n}\{e_i(e_i(e_j(e_j\psi)))-e_i(Je_i(e_j(Je_j\psi)))\}+\textrm{l.o.t.}\\ &=& \Delta^2\psi-\{\sum_{i=1}^{n}(e_{2i-1}Je_{2i-1}+e_{2i}Je_{2i})(\sum_{j=1}^{2n}e_j(Je_j\psi))\}+\textrm{l.o.t.}\\ &=&
\Delta_g^2\psi-\{\sum_{i=1}^{n}(e_{2i-1}e_{2i}-e_{2i}e_{2i-1})(\sum_{j=1}^{2n}e_j(Je_j\psi))\}+\textrm{l.o.t.}\\ &=& \Delta_g^2\psi+\textrm{l.o.t.},
\end{eqnarray*}
where l.o.t denotes the terms of differentiations up to the 3rd order and $\Delta$ is the Euclidean Laplacian.
\end{proof}

\section{Almost-\k Surgery and deformation}


Here we consider a left-invariant almost-\k metric on the 4-dimensional Kodaira-Thurston nil-manifold \cite{Ab}.
The metric can be
written on ${\Bbb R}^4 = \{ (x,y,z,t) | x,y,z,t \in {\Bbb R} \}$
as $$g_{KT} = dx^2 + dy^2 + (dz - x dy)^2 + dt^2$$ and the left-invariant
symplectic form is $\omega_{KT} = dt\wedge dx + dy \wedge dz$.
 The
almost complex structure $J$ is then given by $J(e_4)= e_1$,
$J(e_1)= -e_4$, $J(e_2)= e_3$, $J(e_3)= -e_2$, where
$e_1 =
\frac{\partial}{\partial x}$, $e_2 = \frac{\partial}{\partial y} +
x \frac{\partial}{\partial z}$,
 $e_3 = \frac{\partial}{\partial z}$,
 $e_4 = \frac{\partial}{\partial t}$ which form an orthonormal frame for the metric.

Now we consider a metric $g_n$ on the product manifold ${\Bbb R}^4 \times {\Bbb R}^{2n-4}$, $n \geq 2$ defined
by $g_n = g_{KT} + g_{Euc}$, where $g_{Euc}$ is the Euclidean metric on  ${\Bbb R}^{2n-4}$.
This manifold has the symplectic form $\tilde{\omega} = \omega_{KT} + \omega_0$, where $\omega_0$ is the standard symplectic
structure on ${\Bbb R}^{2n-4}$ and $g_n$ is $\tilde{\omega}$-almost-\k.

\bigskip

Given any symplectic manifold $(M, \omega)$ with an almost-\k metric $g$ and the corresponding almost complex structure $J_g$, we pick a point $p$.
There exists a Darboux coordinate neighborhood $(U, x_i)$ of $p$ with $x(p) =0$ so that
$\omega = \sum_{i=1}^{n} dx_{2i-1} \wedge dx_{2i}$. Assume that $U$ contains a ball $B^g_{\delta}(p)$ of $g$-radius $\delta$ with center at $p$. By considering the (coordinates-rearranging) local diffeomorphism
$\phi(x_1,  \cdots , x_{2n})= (x_2,x_3,x_4,x_1, x_{5} \cdots, x_{2n})$, $\phi:U \rightarrow  {\Bbb R}^4 \times {\Bbb R}^{2n-4}$, i.e. identifying $x_2=x, \ x_3=y, \ x_4= z, \ x_1=t \cdots $, we can get the pulled-back metric of $g_n$ via $\phi$, which we still denote by $g_n$. Note that $\phi^* \tilde{\omega} = \omega$.
We may express $g_n =  g  \cdot e^{h}$ on
$U $ for a unique smooth symmetric
$J_g$-anti-invariant tensor $h$ from (\ref{geh}),  because $ g$ and
$g_n$ are both $\omega$-almost-k\"ahler.
Let
 $\eta (r)$  be a
smooth cutoff function in $C^{\infty}(\Bbb{R}^{>0}, [0, 1])$ s.t.
$\eta \equiv 0$ for $ r<  \frac{\delta}{3} $ and $\eta \equiv 1$ on the set $ r
\geq \frac{2\delta}{3}$.

We define a new $\omega$-almost-\k metric $h$ on $M$ by

\begin{equation} \label{cutoff}
h :=
\begin{cases}
& g \hspace{0.9in} \rm{on }  \ \ \ \ \
M \setminus B^g_{\frac{2\delta}{3}}(p), \\
&g_n  \cdot e^{\eta(r_{g})h}   \hspace{0.3in} \rm{on}  \ \ \ \ \   B^g_{\frac{2\delta}{3}}(p) \setminus B^g_{\frac{\delta}{3}}(p) \\
&  g_n   \hspace{0.8in}  \rm{on} \ \ \ \  B^g_{\frac{\delta}{3}}(p)
\end{cases}
\end{equation}
where $r_{g}$ is the distance function of $g$ from $p$.

\bigskip
In \cite{Ki2}, using a Lohkamp type argument \cite{Lo1} adapted to symplectic manifolds,
 the first author proved on any closed symplectic manifold of dimension$\geq 4$ the existence of an almost-\k metric of negative scalar curvature which equals a prescribed negative scalar-curved metric on an open subset;


\begin{prop} \cite[Theorem 3]{Ki2} \label{pp1}
Let $S$ be a closed  subset in a smooth closed symplectic manifold $(M, \omega)$ of dimension $2n\geq 4$, and $U \supset  S$ be an open neigborhood. Then for any smooth $\omega$-almost-\k metric $g_0$ on $U$ with $s(g_0) <0$,  there exists a smooth $\omega$-almost-\k metric $g$ on $M$ such that $g= g_0$ on $S$ and $s(g) <0$ on $M$.
\end{prop}

We now apply proposition \ref{pp1} to the metric $h$ of (\ref{cutoff}) with $U=B^g_{\frac{\delta}{3}}(p)$ and $S= B^g_{\frac{\delta}{6}}(p)$. We get;

\begin{cor}\label{kim-corollary}
On any smooth closed symplectic manifold $(M, \omega)$ of dimension $2n\geq 4$, there exists a smooth $\omega$-almost-\k metric $g$ with negative scalar curvature such that $g$ is isometric to the product metric of $g_{KH}$ and the Euclidean metric on an open subset $B_\epsilon$ of $M$.
\end{cor}

\section{Proof of Theorem \ref{maint}}
\begin{thm} \label{maina}
Let $(M,\omega)$ be a smooth closed symplectic manifold of dimension $2n\geq 4$. Then any smooth function on $M$  which is somewhere negative and somewhere zero is the scalar curvature of some smooth almost-K\"ahler metric associated to $c\varphi^*\omega$, where $c>0$ is a constant and $\varphi$ is a diffeomorphism  of $M$, isotopic to the identity.
\end{thm}

\begin{proof}
On $(M,\omega)$, let's take the $\omega$-almost-\k metric $g$ constructed in Corollary \ref{kim-corollary}.
By Lemma \ref{Rhu-win} below, we have that $D_gS_\omega$ is surjective. If $f\in C^\infty(M)$ is somewhere negative and somewhere zero,
then for a sufficiently large constant $c>0$,  $$\inf {c}f\leq S_\omega(g)\leq \sup {c}f.$$ Then by Proposition \ref{mainprop},
$cf$ is the scalar curvature of an almost-K\"ahler metric $G$ compatible with $\phi^*\omega$ for some diffeomorphism $\phi$ of $M$ isotopic to the identity, and hence $f$ is the scalar curvature of the almost-K\"ahler metric $c \cdot G$ compatible with $c\phi^*\omega$.
\end{proof}
\noindent Theorem \ref{maint} follows from the above theorem \ref{maina}, since $c\phi^*\omega$ is deformation equivalent to $\omega$ for any constant $c>0$.
\begin{lem}\label{Rhu-win}
For the metric $g$ constructed in Corollary \ref{kim-corollary},  the kernel of  $ \ (D_gS_\omega)^*$
 is $\{ 0 \}$.
\end{lem}
\begin{proof}
We look into the computations in \cite[Section 4]{kim}, where any global solution on Kodaira-Thurston compact manifold was shown to be zero,  whereas we shall now improve to show any {\it local} solution on it is zero.

Let's first compute the curvature of $g$ on $B_\epsilon$. Since it is the product of the Kodaira-Thurston metric and the Euclidean metric, we will only list components for $i=1,\cdots, 4$. (Recall the frame $e_1, e_2, e_3, e_4$ in Section 5.)
From the formula
$$2\langle\nabla_XY,Z\rangle=X\langle Y,Z\rangle+Y\langle X,Z\rangle-Z\langle X,Y\rangle-\langle X,[Y,Z]\rangle-\langle Y,[X,Z]\rangle+\langle Z,[X,Y]\rangle,$$ one can compute
$$\nabla_{e_1}e_2=-\nabla_{e_2}e_1=\frac{1}{2}e_3,\ \ \ \nabla_{e_1}e_3=\nabla_{e_3}e_1=-\frac{1}{2}e_2,$$
$$\nabla_{e_2}e_3=\nabla_{e_3}e_2=\frac{1}{2}e_1,\ \ \ \nabla_{e_i}e_i=\nabla_{e_4}e_i=\nabla_{e_i}e_4=0.$$
Thus letting $$\nabla e_i=\sum_j\omega_{ij}e_j,$$ we have
$$\omega_{ij}=\frac{1}{2}
\left(
  \begin{array}{cccc}
    0 & dz-xdy & dy & 0 \\
    -dz+xdy & 0 & -dx & 0 \\
    -dy & dx & 0 & 0 \\
    0 & 0 & 0 & 0 \\
  \end{array}
\right),
$$
and the Cartan structure equation gives
\begin{eqnarray*}
\Omega_{ij}&=& d\omega_{ij}+\sum_k\omega_{ik}\wedge\omega_{kj}\\ &=& \frac{1}{4}
\left(
  \begin{array}{cccc}
    0 & -3dx\wedge dy & dx\wedge(dz-xdy) & 0 \\
    3dx\wedge dy & 0 & dy\wedge dz & 0 \\
    (dz-xdy)\wedge dx & -dy\wedge dz & 0 & 0 \\
    0 & 0 & 0 & 0 \\
  \end{array}
\right).
\end{eqnarray*}
Denoting the sectional curvature of the plane spanned by $e_i$ and $e_j$ by $K_{ij}$, we have
$$K_{12}=-\frac{3}{4},\ K_{13}=\frac{1}{4},\ K_{23}=\frac{1}{4},\ K_{i4}=0$$ for any $i$.
Then the Ricci tensor $(r_{ij})$ is given by
$$r_{11}=-\frac{1}{2},\ r_{22}=-\frac{1}{2},\ r_{33}=\frac{1}{2},\ r_{44}=0,\ r_{ij}=0 \ \textrm{for } i\ne j$$ so that
$$r_{11}^-=-\frac{1}{4},\ r_{22}^-=-\frac{1}{2},\ r_{33}^-=\frac{1}{2},\ r_{44}^-=\frac{1}{4},\ r_{ij}^-=0 \ \textrm{for } i\ne j.$$

Denoting  $(\nabla d\psi)(e_i,e_j)=e_i(e_j\psi)-(\nabla_{e_i}e_j)\psi$ for $\psi\in C^\infty(B_\epsilon)$ by $\nabla d\psi_{ij}$,  one can easily get
$$\nabla d\psi_{11}=\psi_{xx},\ \nabla d\psi_{22}=\psi_{yy}+2x\psi_{yz}+x^2\psi_{zz},\ \nabla d\psi_{33}=\psi_{zz},\ \nabla d\psi_{44}=\psi_{tt},$$
$$\nabla d\psi_{12}=\psi_{xy}+x\psi_{yz}+\frac{1}{2}\psi_{z},\ \ \ \ \ \nabla d\psi_{13}=\psi_{xz}+\frac{1}{2}\psi_{y}+\frac{x}{2}\psi_{z},\ \ \ \ \  \nabla d\psi_{14}=\psi_{xt},$$
$$\nabla d\psi_{23}=\psi_{yz}+x\psi_{zz}-\frac{1}{2}\psi_{x}, \ \ \ \ \ \ \ \ \nabla d\psi_{24}=\psi_{yt}+x\psi_{zt}, \ \ \ \ \ \ \ \ \ \nabla d\psi_{34}=\psi_{zt}.$$
We list only $\nabla d\psi_{ij}$ for $i,j=1,\cdots,4$, because that's enough for our purpose.

\smallskip
Also denoting $(\nabla d\psi)^-(e_i,e_j)=\frac{1}{2}(\nabla d\psi(e_i,e_j)-\nabla d\psi({J}e_i, {J}e_j))$ simply by $\nabla^-d\psi_{ij}$, one can  get
$$2\nabla^-d\psi_{11}=\psi_{xx}-\psi_{tt}, \ \ \ \ \ \ \ \  \ 2\nabla^-d\psi_{22}=\psi_{yy}+2x\psi_{yz}+(x^2-1)\psi_{zz},$$
$$2\nabla^-d\psi_{12}=\psi_{xy}+x\psi_{xz}+\frac{1}{2}\psi_z+\psi_{zt},\ 2\nabla^-d\psi_{13}=\psi_{xz}+\frac{1}{2}\psi_y+\frac{1}{2}x\psi_z-\psi_{yt}-x\psi_{zt},$$
$$2\nabla^-d\psi_{23}=2(\psi_{yz}+x\psi_{zz}-\frac{1}{2}\psi_x), \ \ \ \  \ \ \ \ \ \ \ \  2\nabla^-d\psi_{14}=2\psi_{xt}.$$

Now suppose $\psi\in \ker (D_gS_\omega)^*$, i.e.
\begin{equation} \label{dsst}
 \nabla^-d\psi-\psi r^-=0.
\end{equation}
 Then from the above, we get the following 6 equations of $\psi$ on $B_\epsilon$ :
\begin{eqnarray}\label{equation1}
\psi_{xx}-\psi_{tt}=-\frac{1}{2}\psi,
\end{eqnarray}
\begin{eqnarray}\label{equation2}
\psi_{yy}+2x\psi_{yz}+(x^2-1)\psi_{zz}=-\psi,
\end{eqnarray}
\begin{eqnarray}\label{equation3}
\psi_{xy}+x\psi_{xz}+\frac{1}{2}\psi_z+\psi_{zt}=0,
\end{eqnarray}
\begin{eqnarray}\label{equation4}
\psi_{xz}+\frac{1}{2}\psi_y+\frac{1}{2}x\psi_z-\psi_{yt}-x\psi_{zt}=0,
\end{eqnarray}
\begin{eqnarray}\label{equation5}
\psi_{yz}+x\psi_{zz}-\frac{1}{2}\psi_x=0,
\end{eqnarray}
\begin{eqnarray}\label{equation6}
\psi_{xt}=0.
\end{eqnarray}

In order to deduce $\psi=0$ on $B_\epsilon$ out of these 6 equations, let's write the local coordinate $(x,y,z,t,x_5,\cdots,x_{2n})$ on $B_\epsilon$ as $(x,y,z,t)\times w$ so that $w=(x_5,\cdots,x_{2n})$. We will first show $\psi_t=0$ and then $\psi_x=0$, which together imply $\psi=0$ by (\ref{equation1}).

From (\ref{equation6}), $\psi(x,y,z,t,w)$ can be written as $a(x,y,z,w)+b(y,z,t,w)$. Substituting it into (\ref{equation1}) gives $$a_{xx}-b_{tt}=-\frac{1}{2}(a+b).$$ Then the LHS of $a_{xx}+\frac{1}{2}a=b_{tt}-\frac{1}{2}b$ is a function of $x,y,z,w$, whereas its RHS is a function of $y,z,t,w$. Thus both sides are functions of y,z, and w only. Differentiating the RHS with respect to $t$ gives $$b_{ttt}-\frac{1}{2}b_t=0.$$ Solving this ODE, we get $$b_t=b_1(y,z,w)e^{\frac{t}{\sqrt{2}}}+b_2(y,z,w)e^{-\frac{t}{\sqrt{2}}}$$ so that
$$b=\sqrt{2}b_1(y,z,w)e^{\frac{t}{\sqrt{2}}}-\sqrt{2}b_2(y,z,w)e^{-\frac{t}{\sqrt{2}}}+b_3(y,z,w).$$
Now plugging $a+b$ into (\ref{equation3}), and picking up only $t$ terms, we get
$$\frac{1}{2}(\sqrt{2}\frac{\partial b_1}{\partial z}e^{\frac{t}{\sqrt{2}}}-\sqrt{2}\frac{\partial b_2}{\partial z}e^{-\frac{t}{\sqrt{2}}})+
\frac{\partial b_1}{\partial z}e^{\frac{t}{\sqrt{2}}}+\frac{\partial b_2}{\partial z}e^{-\frac{t}{\sqrt{2}}}=0$$
so that $\frac{\partial b_1}{\partial z}=\frac{\partial b_2}{\partial z}=0$, and hence we can conclude that $b_1$ and $b_2$ are functions of $y$ and $w$ only.

Plugging new $a+b$ into (\ref{equation4}), and picking up only $t$ terms, we get
$$\frac{1}{2}(\sqrt{2}\frac{\partial b_1}{\partial y}e^{\frac{t}{\sqrt{2}}}-\sqrt{2}\frac{\partial b_2}{\partial y}e^{-\frac{t}{\sqrt{2}}})-(
\frac{\partial b_1}{\partial y}e^{\frac{t}{\sqrt{2}}}+\frac{\partial b_2}{\partial y}e^{-\frac{t}{\sqrt{2}}})=0,$$
which implies that $\frac{\partial b_1}{\partial y}=\frac{\partial b_2}{\partial y}=0$, and hence $b_1$ and $b_2$ are functions of $w$ only.

Again plugging this new $a+b$ into (\ref{equation2}), and picking up only $t$ terms, we get
$$0=-(\sqrt{2}b_1e^{\frac{t}{\sqrt{2}}}-\sqrt{2}b_2e^{-\frac{t}{\sqrt{2}}}),$$ which finally implies that $b_1=b_2=0$ and hence $\psi_t=0$.

Taking $\frac{\partial}{\partial x}$ to (\ref{equation3}) produces $$\psi_{xxy}+x\psi_{xxz}+\frac{3}{2}\psi_{xz}=0.$$
Applying $\psi_{xx}=-\frac{1}{2}\psi$ from (\ref{equation1}), this becomes $$-\frac{1}{2}\psi_y-\frac{1}{2}x\psi_z+\frac{3}{2}\psi_{xz}=0.$$
Comparing it with $\psi_{xz}+\frac{1}{2}\psi_y+\frac{1}{2}x\psi_z=0$ from (\ref{equation4}) gives
$$\psi_y+x\psi_z=0$$ so that $$\psi_{yz}+x\psi_{zz}=0.$$
Combing it with (\ref{equation5}), we get desired $\psi_x=0$.

Finally we have $\psi=0$ on $B_\epsilon$. $\psi$ is a solution of a linear elliptic equation $(D_gS_\omega)\circ (D_gS_\omega)^* \psi=0$ whose principal part is bi-Laplacian from Lemma \ref{kim-idea}. So we get  $\psi=0$ on $M$ by the unique continuation principle for a bi-Laplace type operator\footnote{For example, one can apply Protter's theorem \cite{protter} which states that if a real-valued function $u$ defined in a domain $D\subset \Bbb R^m$ containing $0$ satisfies that $|\Delta^nu|\leq f(x,u,Du,\cdots,D^ku)$ for Lipschitzian $f$ and $k\leq [\frac{3n}{2}]$, and $e^{2r^{-\beta}}u\rightarrow 0$ as $r:=\sqrt{x_1^2+\cdots+x_m^2}\rightarrow 0$ for any constant $\beta>0$, then $u$ vanishes identically in $D$.}, finishing the proof.
\end{proof}

\bigskip

\begin{remark}
{\rm  In our argument, a particular metric $g_{KT}$ on Kodaira-Thurston manifold is used, as it guarantees the surjectivity of the derivative of scalar curvature map at the constructed metric.
 One may guess, reasonably, that a generic almost \k metric satisfies this surjectivity, as in the Riemannian case \cite[4.37]{Be}. However, the equation  (\ref{dsst}) is not readily understood.
For this matter, the article \cite[Theorem 7.4]{BCS} on local deformation of scalar curvature is interesting; generic (local) surjectivity of scalar curvature was shown by some method. In that argument a crucial part was the existence of one real analytic metric {\it without local KIDs}, i.e.  with (locally) surjective derivative of scalar curvature functional.

Here we did not try to prove such generic surjectivity. Rather we bypassed it; we found one almost \k metric without local KIDs, and then we imbedded it onto any symplectic manifold, obtaining a no-global-KID metric.
 We hope to address this generic surjectivity issue in near future.}
\end{remark}

\begin{remark}
{\rm For an almost \k version of Kazdan-Warner theory,   one may try the Hermitian scalar curvature $\frac{1}{2} (s + s^{*})$ \cite{AD} rather than the usual scalar curvature, where $s^*$ is the {\it star scalar} curvature.
 However, although the Hermitian scalar curvature is natural in almost \k geometry,  Kazdan-Warner theory goes better with usual one.
 To see this, recall the  Blair's formula \cite{Bl3}  for $g \in \Omega_{\omega}$:
\begin{equation} \label{Blair}
   \int \frac{1}{2} (s_g + s^{*}_g) \frac{\omega^n}{n!} =  4\pi c_1(\omega)
\cdot \frac{[\omega]^{n-1}}{(n-1)!}.
\end{equation}

Since $c_1$ of any symplectic structure on the 4-d torus is zero by Taubes' result \cite{taubes1},
no negative function can be the Hermitian scalar curvature of an almost \k metric on 4-d torus.}
\end{remark}

\end{document}